\documentclass[11pt]{article}

\usepackage{amsmath,amsfonts,amssymb} \usepackage{bm} \usepackage{graphicx}
\usepackage{mathrsfs} \usepackage{enumitem} \usepackage{booktabs}
\usepackage{hyperref} \usepackage{empheq} \usepackage{xcolor}
\usepackage{microtype} \usepackage{tikz} \usepackage{caption}
\usepackage{subcaption} \usepackage{multirow}

\usepackage{array}
\newcolumntype{?}{!{\vrule width 1.2pt}} \newcolumntype{^}{!{\vrule width
    0.8pt}} \newcommand{\hhrule}{\specialrule{1.2pt}{0pt}{0pt}}
\newcommand{\shrule}{\specialrule{0.8pt}{0pt}{0pt}}

\newtheorem{theorem}{Theorem}[section] \newtheorem{lemma}[theorem]{Lemma}

\newtheorem{corollary}[theorem]{Corollary} \newtheorem{remark}[theorem]{Remark}

\newtheorem{example}{Example}[section]

\numberwithin{equation}{section}

\newenvironment{definition}[1][Definition]{\begin{trivlist}
    \item[\hskip \labelsep {\bfseries #1}]}{\end{trivlist}}

 \def \bb#1{\setbox0=\hbox{$#1$} \kern-.025em\copy0\kern-\wd0
  \kern.05em\copy0\kern-\wd0 \kern-.025em\raise.0433em\box0} \def
\endproof{\vrule height8pt width 5pt depth 0pt}

\hypersetup{ colorlinks = true, urlcolor = blue, allcolors = blue, }

\newcommand{\dx}{\,\mathrm{d}x} 
 \newcommand{\dmu}{\,\mathrm{d}\mu}
\newcommand{\nn}{\nonumber} \newcommand{\csubset}{\subset \subset} \newcommand{\Div}{\nabla \cdot}
\newcommand{\Grad}{\nabla} 

\title{Optimal Local Error Estimates for Finite Element Methods with
  Measure-Valued Sources}

\author{ \setcounter{footnote}{0} Huadong~Gao \footnote{ School of Mathematics
    and Statistics and Hubei Key Laboratory of Engineering Modeling and
    Scientific Computing, Huazhong University of Science and Technology, Wuhan
    430074, P.R. China ({\tt huadong@hust.edu.cn}).} ~~ and ~~ Yuhui~Huang
  \footnote{ School of Mathematics and Statistics, Huazhong University of
    Science and Technology, Wuhan 430074, P.R. China. ({\tt
      yuhui\_huang@hust.edu.cn})} }

\date{}

\begin{document}

\maketitle

% ----------------------------------------------------------------
\begin{abstract}
  We study finite element approximations of second-order elliptic problems with
  measure-valued right-hand sides supported on lower-dimensional sets. The exact
  solution generally lacks $H^1$-regularity due to the source singularity, which
  limits global convergence rates of numerical methods. Using a very weak
  solution framework, we establish well-posedness and global error estimates for
  standard Lagrange finite element methods on Lipschitz polyhedral/polygonal
  domains. By using interior estimates techniques, we prove optimal local $L^2$-
  and $H^1$-error estimates in subdomains that are strictly separated from the
  support of the measure. Extensive numerical experiments are provided to verify
  the theoretical results. These results show that for Lagrange FEMs solving
  elliptic problems with singular right-hand sides, the loss of global
  convergence is purely local, and that optimal convergence rates still hold
  away from the singular source.

\end{abstract}

% ----------------------------------------------------------------
\section{Introduction}

We consider the elliptic boundary value problem
\begin{equation}
  -\Div \bigl(A(\bm{x})\Grad u\bigr) = \mu \quad \text{in } \Omega \,,
  \qquad
  u = 0 \quad \text{on } \partial\Omega\,,
  \label{pde}
\end{equation}
where $A(\bm{x}) \in W^{1,\infty}(\Omega)^{d \times d}$ is symmetric and uniformly elliptic, and
the source term $\mu$ is a Radon measure on $\Omega$ with compact support. Here
$\Omega \subset \mathbb{R}^d$ ($d=2,3$) is a bounded Lipschitz polyhedral domain (or
polygonal domain in 2D). Without loss of generality, we denote by
$\mathcal{S} := \operatorname{supp}(\mu)$ the support of $\mu$ and assume that
$\mathcal{S} \csubset \Omega$.

Second-order elliptic equations with localized sources occur in a wide range of
applications, including electrostatics, heat conduction, subsurface flow, and
biological transport. Typical examples include point sources modeled by Dirac
measures, line sources supported on curves, surface sources concentrated on
interfaces, and boundary-supported data of low regularity. In all these
situations, the forcing term is naturally described by a Radon measure whose
support may have lower dimension than the ambient domain. Unlike formulations
restricted to specific source types, this measure-theoretic approach provides a
unified framework that naturally accommodates discrete, singular, and
distributed sources. The presence of singular data generally prevents the
solution from belonging to $H^1(\Omega)$ and leads to reduced convergence rates for
standard finite element methods. As a consequence, the classical weak
formulation in $H^1(\Omega)$ is no longer applicable, which motivates the use of very
weak formulations and specialized finite element techniques for elliptic
problems with measure-valued right-hand sides.

Mathematical analysis for the above model problem is the focus of the monograph
\cite{Ponce2016} by Ponce. Finite element approximations of elliptic problems with
low-regularity sources have been studied extensively. One popular approach for
Poisson-type equations with measure-valued data is based on the
$(W_0^{1,p}, W_0^{1,q})$ variational framework, where $W_0^{1,q}$ is embedded
into a space of continuous functions, so that the test functions are well
defined for the measure $\mu$, see \cite{Apel-Vexler2011,Araya-2006,Bertoluzza2018,Koppl-Wohlmuth2014,Lee-Choi2020,Masri2023}. Among these works, \cite{Apel-Vexler2011,Araya-2006} focus primarily on
global error estimates or mesh grading near singularities. Local error estimates
on subdomains away from the support of the measure are derived in \cite{Bertoluzza2018,Koppl-Wohlmuth2014,Koppl-VW2016,Lee-Choi2020,Masri2023}. It should
be noted, however, that to obtain higher-order local convergence rates, the
analysis in \cite{Bertoluzza2018,Masri2023} require $\Omega$ to be smooth or to have a very specific geometry,
such as a rectangle or cube. We shall point out that the main results and
analysis in \cite{Lee-Choi2020} are in fact not correct, since the regularity result stated in
Lemma~2.1 does not hold for general convex polygonal or polyhedral domains, see
our numerical results in Example~\ref{example:hexagon-domain} of Section~\ref{numerical-experiments}.

An alternative and more robust framework, based on the notion of very weak
solutions (see \cite{Casas1985} and \cite[Section~2.3]{Vexler2025}), enables the systematic analysis of
Poisson's equations with arbitrary singular sources. In \cite{Berggren2004}, Berggren proposed a
very weak FEM with linear element to solve the very weak solution. In this
paper, we adopt the very weak solution framework to analyze finite element
approximations of the above model problem. We show that Berggren's very weak
finite element method is equivalent to the standard conforming finite element
method for $k$th-order Lagrange elements, for any $k \ge 1$. Based on this
equivalence, we derive global error estimates on general Lipschitz polygonal and
polyhedral domains. Furthermore, we establish optimal local error estimates on
subdomains that are strictly separated from the support of the measure, without
resorting to mesh grading or local refinement. For several special classes of
domains, we also show that higher-order local convergence can be achieved. Our
analysis follows the classical philosophy of local estimates initiated by
Nitsche, Schatz and Wahlbin \cite{Nitsche-Schatz74,Schatz-Wahlbin1977,Walbin-Lars}, with interior regularity results playing a
crucial role. In addition, we present extensive numerical experiments
demonstrating that the derived global and local error estimates are sharp up to
the regularity of the exact solution. In particular, our numerical results
indicate that for general polyhedral/polygonal domains, corner singularities may
dominate the convergence behavior even when the source term is located away from
the boundary. For Lagrange FEMs, the singularity of the right-hand side does not
induce pollution effects, in the sense that it does not deteriorate the
numerical solution in regions away from the source. In contrast, corner
singularities are much more severe and typically lead to pollution over a large
portion of the computational domain, thereby reducing both global and local
convergence accuracy. We also note that the monograph by Vexler and Meidner \cite{Vexler2025}
provides both global and local error estimates for linear finite element
approximations of very weak solutions on convex domains.

The rest of the paper is organized as follows. In Section~\ref{sect-Preliminaries}, some preliminary
results and auxiliary lemmas are introduced. In Section~\ref{sect-method}, we introduce the
very weak formulation and its finite element discretization and derive global
error estimates. In Section~\ref{sect-localestimate}, we establish an improved local error estimates
on subdomains away from the support of the source. In Section~\ref{numerical-experiments}, we present
extensive numerical experiments to confirm the theoretical results. Finally,
concluding remarks are given in Section~\ref{sect-Conclusion}.

% ----------------------------------------------------------------
\section{Preliminaries}
\label{sect-Preliminaries}

Before we present our result, we shall introduce some notations. For any two
functions $u$, $v \in {L}^{2}(\Omega)$, we denote the ${L}^{2}(\Omega)$ inner product and the
$L^2$ norm by
\begin{equation}
  (u,v) = \,
  \int_{\Omega} u(\bm{x}) \, v(\bm{x}) \, {\mathrm{d}} \bm{x} ,
  \qquad {\left \| u \right \|_{L^2(\Omega)}} = (u,u)^{\frac{1}{2}} \,.
\end{equation}
Let $W^{k,p}(\Omega)$ be the Sobolev space defined on $\Omega$, and ${W}^{k,p}_0(\Omega)$ be
the subspace of $W^{k,p}(\Omega)$ with zero trace. By conventional notations, we
define $H^{k}(\Omega) := W^{k,2}(\Omega)$ and ${H}^{k}_0(\Omega) := {W}^{k,2}_0(\Omega)$. For a
positive real number $s= k + w$, with $ w \in (0,1)$, we define
$H^s(\Omega)=(H^{k}(\Omega), H^{k+1}(\Omega))_{[w]}$ via the complex interpolation, see \cite[Theorem
6.4.5]{Bergh}. To abbreviate notations, we adopt $\| \cdot \|_{L^2(\Omega)}$ and
$\| \cdot \|_{H^r(\Omega)}$ for the standard $L^2$ and $H^r$ norms on the domain $\Omega$. We
also define the bilinear form $a(\cdot,\cdot) : H^{1}(\Omega) \times H^{1}(\Omega) \to \mathbb{R}$ by
\[
  a(u,v) = \int_{\Omega} A(\bm{x}) \Grad u \cdot \Grad v \mathrm{d} \bm{x},
\]
where $A(\bm{x})$ is defined as in \eqref{pde}.

To accommodate the measure $\mu$, we introduce the space $\mathcal{M}(\Omega)$ of
finite Radon measures on $\Omega$. This space is endowed with the total variation
norm
\[
  \|\mu\|_{\mathcal{M}(\Omega)} := |\mu|(\Omega) = \sup\left\{ \int_\Omega \phi \, \mathrm{d}\mu : \phi \in C_0(\Omega), \|\phi\|_{L^\infty(\Omega)} \leq 1 \right\},
\]
where $C_{0}(\Omega)$ denotes the space of continuous functions with compact support.
Equipped with this norm, $\mathcal{M}(\Omega)$ is a Banach space. One can prove that
a Radon measure $\mu$ with compact support is finite, i.e. $\mu \in \mathcal{M}(\Omega)$.
For more details, we refer the readers to \cite{Ponce2016}.

To discretize the problem \eqref{pde}, let $\{\mathcal{T}_h\}$ be a shape-regular family
of triangulations of $\Omega$. For $k\ge1$, we define the Lagrange finite element
space
\[
  V_h := \{ v_h\in H^1(\Omega): v_h|_K \in \mathbb{P}_k(K), \ \forall K\in\mathcal{T}_h\},
\]
and denote by
\[
  V_{h,0}:=V_h\cap H_0^1(\Omega)
\]
the subspace of functions with vanishing trace on $\partial \Omega$. We use $\Pi_h$ to denote
the general Lagrange interpolation, which satisfies
\begin{align}
  \label{Lagrange-interp-error}
  \| v - \Pi_h v\|_{H^r(D)} &\le C h^{s-r} |v|_{H^{s}(D')}, \quad \textrm{$r \le s \le k+1 $},\\
  \| v - \Pi_h v\|_{W^{r,\infty}(D)} &\le C h^{s-r-\frac{d}{p}} |v|_{W^{s,p}(D')}, \quad \textrm{$r \le s \le k+1 $},\label{eq:max-interp}
\end{align}
where $r = 0,1$, $D \subset \Omega$, and $D'$ is the union of all mesh elements that
intersect $D$, i.e. $D' = \bigcup\{T \in \mathcal{T}_h : T \cap D \neq \emptyset\}$, see \cite[Corollary~4.4.24]{Brenner-Scott2008}
for more details.

In the rest of this paper, we use a unified index $s$ to describe the regularity
of the standard Poisson's equation with homogeneous Dirichlet boundary condition
on the domain $\Omega$.
\begin{lemma}
  \label{lemma-regularity}
  Let $f\in L^2(\Omega)$, the coefficient matrix $A \in W^{1,\infty}(\Omega)^{d \times d}$ be symmetric
  and uniformly elliptic, then the solution $u$ to the following problem
  \begin{equation}
    \begin{cases}
      -\Div (A(\bm{x})\Grad u) &= f \,, \quad \quad \mathrm{in} ~ \Omega\,,  \\
      \hfill u &= 0\,, \quad\quad  \mathrm{on}~ \partial \Omega\,,
    \end{cases}
    \label{poisson-standard}
  \end{equation}
  satisfies
  \begin{align}
    \|u\|_{H^{1+s}(\Omega)} \le C \|f\|_{L^2(\Omega)}, \quad
    \textrm{with} \quad
    \left\{
    \begin{array}{ll}
      s = 1,
      & \textrm{$\Omega$ is convex,}
      \\[2pt]
      s = \sup \{ \lambda_{\textrm{2D}} \} - \epsilon,
      & \textrm{$\Omega$ is non-convex in 2D},
      \\[2pt]
      s = \sup \{ \lambda_{\textrm{3D}} \} - \epsilon,
      & \textrm{$\Omega$ is non-convex in 3D},
    \end{array}
    \right.
    \label{index-unified}
  \end{align}
  where $\lambda_{\textrm{2D}}=\frac{\pi}{\max_j \Theta_j} > \frac{1}{2} $ with $\{\Theta_j\}$ denoting the
  re-entrant interior angles of $\Omega$, and $\lambda_{\textrm{3D}} > \frac{1}{2}$ which depends on
  both edges and corners of the three-dimensional polyhedron, respectively. Here
  $\epsilon>0$ is any arbitrarily small number. For Lipschitz polygonal/polyhedral
  domains, the solution $u$ to \eqref{poisson-standard} always belongs to $H^{3/2+\epsilon}(\Omega)$, see \cite[Theorem
  3.1]{Berggren2004}.
\end{lemma}
% ----------------------------------------------------------------

The following lemma concerns the regularity of Poisson's equation with interior
source, which implies that the global regularity of $u$ can be higher than
$H^2(\Omega)$ in several cases, see \cite[Theorem 5.1.2.4]{Grisvard1985} and \cite{Nitsche-Schatz74}.
\begin{lemma}[Regularity of the solution with an interior source]
  \label{lemma-interior-source}
  Let $\operatorname{supp}(f) \csubset D \subset \Omega$ and assume the source term
  $f \in H^p_0(D)$ with the integer index $p \ge 0$. Then, the solution $u$ to the
  Poisson equation with homogeneous Dirichlet boundary condition
  \begin{equation}
    \begin{cases}
      -\Div (A(\bm{x})\Grad u) &= f \,, \quad \quad \mathrm{in} ~ \Omega\,,  \\
      \hfill u &= 0\,, \quad\quad  \mathrm{on}~ \partial \Omega\,,
    \end{cases}
    \label{poisson-interior-source}
  \end{equation}
  satisfies
  \begin{align}
    \|u\|_{H^{3}(\Omega)} \le C \|f\|_{H^1(\Omega)},
    \label{index-interior-source2}
  \end{align}
  if the domain $\Omega$ is a convex polygon with all interior angles less than or
  equal to $\pi/2$ in two-dimensional space, or $\Omega$ is a convex polyhedron with
  all interior edge angles less than or equal to $\pi/2$. Furthermore, if $\Omega$ is a
  rectangle or equilateral triangle in 2D or a cube in 3D, then there holds
  \begin{align}
    \|u\|_{H^{p+2}(\Omega)} \le C \|f\|_{H^p(\Omega)}, \qquad \textrm{for any $p \ge 0$.}
    \label{index-interior-source3}
  \end{align}
  Where $C > 0$ depends only on $A$, $p$, and $\Omega$.
\end{lemma}

% ----------------------------------------------------------------
The higher interior regularity estimates on subdomains away from the source
follow from standard results for harmonic functions, see \cite[Section~6.3]{Evans}.

\begin{lemma}[Interior regularity away from the source]
  \label{lemma-evans}
  Let $\Omega \subset \mathbb{R}^d$ be a bounded domain and $u \in H_0^1(\Omega)$ satisfies
  \[
    a(u,v) = (f,v) ,\quad \forall v \in H_{0}^{1}(\Omega).
  \]
  Suppose that $f \in L^2(\Omega)$ and $f\equiv 0$ in an open set $D_1 \csubset \Omega$. Then for
  any $D_0 \csubset D_1$ and any integer $k\ge 0$, there holds
  \begin{align}
    \|u\|_{H^{k+1}(D_0)}\le C \|u\|_{L^2(D_{1})},
    \label{evans-estimate-origin}
  \end{align}
  where the constant $C>0$ depends on $k$, $A$ , $D_0$, and $D_1$, but is
  independent of $u$. If $\partial \Omega$ is Lipschitz, by global regularity we have the
  estimate
  \begin{equation}
    \|u\|_{H^{k+1}(D_0)}\le C \|f\|_{L^{2}(\Omega)},
    \label{eq:evans-estimate}
  \end{equation}
  for any integer $k \ge 0$, where $C > 0$ is independent of $u$.
\end{lemma}
% ----------------------------------------------------------------

We present several useful results in the following lemmas, which will be
frequently used in our proof. The local $L^2$ and energy estimates can be found
in \cite[Theorem 5.1]{Nitsche-Schatz74} and \cite[Theorem 3.4]{Demlow-others-2011}, while the local maximum norm estimate
can be found in \cite{Schatz-Wahlbin1977}.
\begin{lemma}[Local $L^2$ and energy estimate]
  \label{lemma-local-energy-Demlow}
  Assume that $\mathcal{T}_h$ is a quasi-uniform mesh partition of a domain
  $\Omega \subset \mathbb{R}^n$. Let $D$ be a subdomain of $\Omega$. For some $d>0$, we define
  $D_d:= \{x \in \Omega; \textrm{dist}(\bm{x},D) \le d\}$ such that $D_d \csubset \Omega$. If
  $u \in H^1_{loc}(\Omega)$ and $u_h$ satisfies
  \begin{align}
    \label{weak-and-regularized-pre}
    \int_{\Omega}A(\bm{x}) \nabla u_h \cdot \nabla v_h \mathrm{d}\bm{x}
    = \int_{D_d}A(\bm{x}) \nabla u \cdot \nabla v_h \mathrm{d}\bm{x},
    \quad \forall v_h \in V_{h,0}, ~ v_h=0 ~ \textrm{on}~ \Omega\backslash D_d \, .
  \end{align}
  Then, for $s=0, 1$ and sufficiently small $h>0$ there holds
  \begin{equation}
    \label{local-energy-estimate2011}
    \|u - u_h\|_{H^s(D)} \leq C \left(h^{1-s}\inf_{\chi_{h}\in V_{h}} \|u - \chi_{h}\|_{H^{1}(D_d)}
      + \|u-u_{h}\|_{L^{2}(D_d)} \right),
  \end{equation}
  where $C>0$ is a constant independent of $u$ and the (local) mesh size $h$. It
  should be pointed out that the local energy estimate, i.e., the case $s=0$,
  holds for general shape-regular meshes.
\end{lemma}
\begin{lemma}[Local maximum norm estimate]
  \label{lemma-Whahlbin}
  Consider the standard elliptic problem
  \begin{equation}
    \label{pde-good}
    -\nabla \cdot (A(\bm{x})\nabla u) = f\quad \text{in } \Omega,
    \qquad
    u=0 \quad \text{on } \partial\Omega,
  \end{equation}
  where $f \in L^2(\Omega)$. A standard FEM is to seek $u_h \in V_{h,0}$, such that
  \begin{align}
    a(u_h,\omega_h) = (f,\omega_h), \quad \forall \omega_h \in V_{h,0}\,.
    \label{fem-good}
  \end{align}
  Let $D_0\csubset D_1\csubset\Omega$. Then, for any integer $p>0$ and sufficiently
  small $h>0$, there exists a constant $C>0$, independent of $h$, such that
  \begin{align}
    \|u-u_h\|_{L^\infty(D_0)}
    & \le C \left( \left|\ln{h}\right|^{\delta^{k,1}} \inf_{\chi_{h}\in V_{h}} \|u - \chi_{h}\|_{L^{\infty}(D_1)}
      + \|u-u_h\|_{H^{-p}(D_1)}  \right)
      \nn    \\
    & \le C \left( h^{k+1}\left|\ln{h}\right|^{\delta^{k,1}} \|u\|_{W^{k+1,\infty}(D_1)}
      + \|u-u_h\|_{H^{-p}(D_1)}  \right).
      \label{whahlbin-estimate}
  \end{align}
  where $\delta^{k,1}$ denote the Kronecker delta.
\end{lemma}

% TODO all need change
% ----------------------------------------------------------------
\section{Finite element approximation and global estimates}
\label{sect-method}

\subsection{The very weak solution and a FEM}
Since the right-hand side $\mu$ does not belong to $H^{-1}(\Omega)$, the standard weak
formulation is not applicable. We therefore adopt a very weak solution
framework.

Define
\begin{equation}
  V := H_0^1(\Omega)\cap \{ v\in L^2(\Omega):\Div (A \Grad v)\in L^2(\Omega)\}.
\end{equation}
Endowed with the norm $\|v\|_V :=\|\Div (A\Grad v)\|_{L^2(\Omega)}$, $V$ is a Hilbert
space.

\begin{lemma}[Well-posedness]\label{lem:very-weak}
  Let $\mu\in\mathcal{M}(\Omega)$. There exists a unique solution $u\in L^2(\Omega)$ such that
  \begin{equation}
    \label{eq:veryweak}
    - \int_\Omega u\, \nabla\cdot(A\nabla v)\,\dx
    = \int_\Omega v\,\dmu,
    \qquad \forall v\in V.
  \end{equation}
  Moreover,
  \[
    \|u\|_{L^2(\Omega)} \le C \|\mu\|_{\mathcal{M}(\Omega)}.
  \]
\end{lemma}
\noindent {\bf Proof.}
By Sobolev embedding (for $d\le 3$), $V\hookrightarrow C^0(\overline{\Omega})$. Hence, any finite
measure $\mu$ defines a continuous linear functional on $V$. Following a procedure
analogous to that in the proof of \cite[Lemma~2.3]{Apel}, we obtain the existence and
uniqueness of the very weak solution \eqref{eq:veryweak} via the Babu\^{s}ka–Lax–Milgram
theorem. \endproof

Now we introduce a finite element method to solve the very weak solution $u$,
following the approach first proposed by Berggren in \cite{Berggren2004}. Compared to the
standard FEM seeking a numerical solution to the weak formulation, Berggren's
scheme seeks a numerical solution to the very weak formulation.

\begin{definition}[(A very weak FEM of Berggren)]
  Seek $u_h\in V_{h,0}$, such that
  \begin{equation}
    (u_h,v_h)=\int_\Omega z_h(v_h)\,\dmu,
    \qquad \forall v_h\in V_{h,0}.
    \label{fem-veryweak1}
  \end{equation}
  where $z_h(v_h)\in V_{h,0}$ is defined by
  \begin{equation}
    a( z_h(v_h), \omega_h) = (v_h,\omega_h),
    \qquad \forall \omega_h\in V_{h,0}.
    \label{fem-veryweak2}
  \end{equation}
\end{definition}

Berggren showed in \cite{Berggren2004} that, for the linear element method the above scheme is
equivalent to the conventional Lagrange FEM, provided
$u_h|_{\scriptstyle \partial \Omega}=P_h^{\scriptstyle \partial \Omega}g$ is used for the boundary data
$g \in L^2(\partial \Omega)$. Here, we point out that for any $k$th-order Lagrange element,
this equivalence also holds. The proof is rather simple for homogeneous
Dirichlet boundary condition.

\begin{theorem}
  Berggren's scheme is equivalent to the following formulation: find
  $u_h\in V_{h,0}$ such that
  \begin{equation}
    a( u_h, \omega_h)=\int_\Omega \omega_h\,\dmu,
    \qquad \forall \omega_h\in V_{h,0}.
    \label{true-FEM}
  \end{equation}
\end{theorem}
\noindent {\bf Proof.} Assume $u_{h}$ is the numerical solution in \eqref{fem-veryweak1}. Taking
$\omega_h= u_h$ into \eqref{fem-veryweak2} gives
\begin{equation}
  a( z_h(v_h), u_h) = \int_\Omega z_h(v_h)\,\dmu\,,
  \qquad \forall v_h\in V_{h,0}\,,
  \label{fem-veryweak1-1}
\end{equation}
where $z_h(v_h)\in V_{h,0}$ is defined by
\begin{equation}
  a(z_h(v_h), \omega_h) = (v_h,\omega_h),
  \qquad \forall \omega_h\in V_{h,0}.
  \label{map-veryweak}
\end{equation}
Clearly, the mapping $v_h \mapsto z_h(v_h)$ defined by \eqref{map-veryweak} is bijective as a mapping
$V_{h,0} \longmapsto V_{h,0}$. Thus, the desired result follows. The converse follows by
an analogous argument. \endproof

\subsection{Global estimates on general Lipschitz polygonal/polyhedral domains}
For the model problem \eqref{pde}, we have the following $L^2$ estimate. This result for
convex domain $\Omega$ is given by Casas in \cite{Casas1985} using a similar approach.

\begin{theorem}[Global error estimate]
  \label{thm:global-error}
  Let $u$ be the very weak solution defined in \eqref{eq:veryweak} and $u_h$ be the numerical
  solutions of \eqref{fem-veryweak1}, respectively. Then, if $\Omega$ is a convex Lipschitz
  polygonal/polyhedral domain or non-convex Lipschitz polygonal/polyhedral
  domain, there holds
  \begin{equation}
    \|u-u_h\|_{L^2(\Omega)}  \le C h^{2-\frac{d}{2}}\|\mu\|_{\mathcal{M}(\Omega)},
    \label{eq::global-first}
  \end{equation}
  except in the case of the linear element method (i.e., $k=1$) on a non‑convex
  Lipschitz polygonal domain, where we have
  \begin{equation}
    \|u-u_h\|_{L^2(\Omega)}  \le C\left(1+ |\ln h|\right)^{1/2} h \|\mu\|_{\mathcal{M}(\Omega)}.
    \label{eq:global-second}
  \end{equation}
  The constant $C$ is independent of $h$.
\end{theorem}
\noindent {\bf Proof.}
For convex domain, we refer to \cite[Theorem 3]{Casas1985}. Therefore, we only consider the
non-convex case. Let $\zeta$ be the solution of the following dual problem
\begin{align}
  -\Div (A \Grad \zeta )=u-u_{h}, \quad \textrm{in $\bm{x} \in \Omega$},
  \qquad
  \zeta=0~ \textrm{ on $\partial \Omega$}.
  \label{dual-pde}
\end{align}
A Lagrange FEM for this dual problem is to seek $\zeta_h \in V_{h,0}$ such that
\begin{align}
  a(\zeta_h, \omega_h) = (u-u_{h,}\omega_h), \quad \forall \omega_h \in V_{h,0}\,.
  \label{dual-fem}
\end{align}
Then we have
\begin{align}
  \|u-u_{h}\|_{L^{2}(\Omega)}^{2} & = -(u,\Div (A \Grad \zeta)) + (u_{h},\Div (A \Grad \zeta)) \nn              \\
  \textrm{(by \eqref{eq:veryweak}) \qquad}
                           & = \int_\Omega \zeta \,\dmu + (u_{h},\Div (A \Grad \zeta))
                             \nn                                                                                                           \\
  \textrm{(by integration by parts) \qquad}
                           & = \int_\Omega \zeta \,\dmu - a(u_h, \zeta)
                             \nn                                                                                                           \\
  \textrm{(by Galerkin orthogonality) \qquad}
                           & = \int_\Omega \zeta  \,\dmu - a(u_h, \zeta_h)
                             \nn                                                                                                           \\
  \textrm{(by \eqref{true-FEM}) \qquad}
                           & = \int_\Omega \zeta  \,\dmu - \int_\Omega \zeta_h\,\dmu
                             \nn                                                                                                           \\
                           & \le \|\mu\|_{\mathcal{M}(\Omega)} \|\zeta-\zeta_h\|_{L^\infty(\mathcal{S})}
                             \label{global-snow}
\end{align}

Now we shall estimate the term $\|\zeta-\zeta_h\|_{L^\infty(\mathcal{S})}$, which is done by
the local $L^2$/$H^1$ and the maximum estimate technique in Lemma~\ref{lemma-Whahlbin}. Assume
$\mathcal{S} \csubset D_{1} \subset D_{1}^{'} \csubset D_{2} \csubset \Omega$, then for
quadratic and higher-order finite elements, the following estimate holds
\begin{align}
  \|\zeta-\zeta_h\|_{L^\infty(\mathcal{S})}
  & \le \|\zeta-\Pi_{h}\zeta\|_{L^{\infty}(D_{1})} + \|\zeta-\zeta_{h}\|_{L^{2}(D_{1})} \nn    \\
  & \le C h^{2-\frac{d}{2}} \|\zeta\|_{H^{2}(D_{2})} + C h^{2s} \|\zeta\|_{H^{1+s}(\Omega)} \nn \\
  & \le C h^{2-\frac{d}{2}} \|u-u_{h}\|_{L^{2}(\Omega)},
    \label{zeta-estimate-first}
\end{align}
where we have noted that the regularity index $s \in (1/2,1)$ for non-convex
Lipschitz domains, the interior $H^2$ regularity of $\zeta$ by Lemma \ref{lemma-evans}, and the
interpolation result \eqref{eq:max-interp}. Therefore, by taking the above estimate \eqref{zeta-estimate-first} into \eqref{global-snow} ,
we obtain
\begin{align}
  \|u-u_h\|_{L^2(\Omega)} \le C h^{2-\frac{d}{2}} \|\mu\|_{\mathcal{M}(\Omega)},
  \label{eq:global-proof-first}
\end{align}
for quadratic and higher-order finite elements.

Next we consider the estimate for linear FEM. Recall that the following discrete
Sobolev inequalities \cite[Section~4.9]{Brenner-Scott2008} hold for any $v_h \in V_{h,0}$
\begin{align}
  \|v_h\|_{L^{\infty}(\Omega)} \le C \, \gamma_{d,h} \, \|v_h\|_{H^{1}(\Omega)},
  \qquad \textrm{with} \quad
  \gamma_{d,h}: =
  \begin{cases}
    (1+ |\ln h|)^{1/2}\,,                   & \text{for } d=2,\\[0.3em]
    h^{-\frac{1}{2}}\,,   & \text{for } d=3.
  \end{cases}
\end{align}
For linear FEM, we estimate local maximum error of $\zeta$ in the following way
\begin{align}
  \|\zeta-\zeta_h\|_{L^\infty(\mathcal{S})}
  & \le   \|\zeta-\Pi_h\zeta\|_{L^\infty(\mathcal{S})} +
    \|\Pi_h\zeta - \zeta_h\|_{L^\infty(\mathcal{S})}
    \nn                                                                                \\
  \textrm{(By inverse inequality)\qquad}
  & \le   C h^{2-\frac{d}{2}} \|\zeta\|_{H^2(D_1)} +
    C \gamma_{d,h}\|\Pi_h\zeta - \zeta_h\|_{H^{1}(\mathcal{S})}
    \nn                                                                                \\
  \textrm{(By Lemma~\ref{local-energy-estimate2011}) \qquad}
  & \le C h^{2-\frac{d}{2}} \|\zeta\|_{H^2(D_1)}
    + C \gamma_{d,h}( \|\zeta - \Pi_h\zeta\|_{H^{1}(D_{1})}
    + \|\zeta - \zeta_h\|_{L^2(D_1 )})
    \nn                                                                                \\
  & \le C h^{2-\frac{d}{2}} \|\zeta\|_{H^2({D_1})}
    + C \gamma_{d,h} ( h\|\zeta\|_{H^{2}(D_{2})}
    + h^{2s}\|\zeta\|_{H^{1+s}(\Omega)})
    \nn                                                                                \\
  & \le C h^{2-\frac{d}{2}} \|\zeta\|_{H^2({D_1})}
    + C \gamma_{d,h} h \| \zeta\|_{H^{1+s}(\Omega)}
    \nn                                                                                \\
  & \le C \left( h^{2-\frac{d}{2}} + \gamma_{d,h} h \right) \| \zeta\|_{H^{1+s}(\Omega)}
    \nn                                                                                \\
  & \le C \gamma_{d,h}  h \|u-u_{h}\|_{L^2({\Omega})} \,.
    \label{zeta-estimate-second}
\end{align}
From \eqref{zeta-estimate-second} we can deduce that for linear FEM
\begin{align}
  \|u-u_h\|_{L^2(\Omega)} \le
  \begin{cases}
    C (1+ |\ln h|)^{1/2} \, h \, \|\mu\|_{\mathcal{M}(\Omega)}, & \text{for } d=2,\\[0.3em]
    C h^{\frac{1}{2}} \|\mu\|_{\mathcal{M}(\Omega)}, & \text{for } d=3.
  \end{cases}  
  \label{eq:global-proof-second}
\end{align}
Combining \eqref{eq:global-proof-first} and \eqref{eq:global-proof-second} we obtain the desired result. \endproof

\begin{remark}
  \rm The convergence rates are independent of the polynomial degree $k$. The
  only exception is an additional factor $|\ln h|^{1/2}$ appearing for linear
  finite elements ($k=1$) on non-convex polygons; this factor has little effect
  in numerical experiments. The above $L^2(\Omega)$ error estimate is sharp, as
  verified numerically in several previous works for convex domains. The
  convexity assumption on the domain $\Omega$ is not essential: the proof relies only
  on interior $H^2$-regularity of the dual solution $\zeta$. Moreover, the
  regularity of the very weak solution $u$ does not enter explicitly into the
  argument. Whether the logarithmic factor $|\ln h|^{1/2}$ can be removed in the
  two-dimensional non-convex polygonal case remains open.

\end{remark}

% ----------------------------------------------------------------
\section{An improved local error estimate on subdomains away from source}
\label{sect-localestimate}

In this section, we prove that the FEM \eqref{true-FEM} has a quasi-optimal convergence rate
on regions away from the support of measure-valued sources. In other words,
there is almost no pollution in finite element solution for singular sources on
general domains. Recall that $\mathcal{S}:=\operatorname{supp}(\mu)$ and
$\mathcal{S}\csubset\Omega$. For a fixed $r>0$, we introduce the subdomain $\Omega_r \subset \Omega$
by excluding an $r$-neighborhood of $\mathcal{S}$,
\[
  \Omega_r := \{x \in \Omega : \operatorname{dist}(x,\mathcal{S}) > r\}.
\]

We present the main result on local errors on subdomain $\Omega_r$ in the following
theorem, which indicates that the reduced global convergence rates are entirely
due to the local singularity of the solution in the neighborhood of
$\mathcal{S}$.

\begin{theorem}
  \label{main-thm-local}
  There exists a constant $C>0$, independent of $h$, such that for sufficiently
  small $h>0$, the following local error estimates hold:
  \begin{empheq}[left=\empheqlbrace]{align}
    & \|u-u_h\|_{L^2(\Omega_r)} \le C \left(h^{k+1}|\ln h|^{\delta^{k,1}} + h^{2s}\right) \|\mu\|_{\mathcal{M}(\Omega)},
    \label{local-L2-linear}
    \\
    & \|u-u_h\|_{H^1(\Omega_r)} \le C h^{\min\{k,s\}} \|\mu\|_{\mathcal{M}(\Omega)}.
    \label{local-H1-linear}
  \end{empheq}
  Here, $s$ is the regularity index defined in \eqref{index-unified}, and $k$ denotes the
  polynomial degree of the Lagrange finite element space $V_h$.
\end{theorem}

\noindent {\bf Proof.}
Let $\xi$ be the solution of the following elliptic problem with homogeneous
Dirichlet boundary condition
\begin{align}
  -\Div(A\Grad \xi) =(u-u_h) \, {\mathcal{X}}_{\scriptscriptstyle \Omega_r}, \quad \textrm{for $\bm{x} \in \Omega$},
  \qquad
  \xi = 0~ \textrm{ on $\partial \Omega$},
  \label{dual-pde-local}
\end{align}
where the characteristic function is defined by
\[
  \mathcal{X}_{\Omega_{r}}=
  \begin{cases}
    1, & \bm{x} \in \Omega_{r},                  \\
    0, & \bm{x} \in \Omega \setminus \Omega_{r}.
  \end{cases}
\]
A Lagrange FEM for this dual problem is to seek $\xi_h \in V_{h,0}$ such that
\begin{align}
  a( \xi_h, \omega_h) = ((u-u_h) \, {\mathcal{X}}_{\scriptscriptstyle \Omega_r} \, , \, \omega_h),
  \quad \forall \omega_h \in V_{h,0}\,.
  \label{dual-fem-local}
\end{align}

By multiplying both side of \eqref{dual-fem-local} with $u-u_h$ and applying a technique similar to
the one used in \eqref{global-snow}, the local error admits the bound
\begin{align}
  \|u-u_h\|_{L^2(\Omega_r)}^2 & = -(u,\Div (A \Grad \zeta)) + (u_{h},\Div (A \Grad \zeta)) \nn                \\
                       & \le \|\mu\|_{\mathcal{M}(\Omega)} \|\zeta-\zeta_h\|_{L^\infty(\mathcal{S})} .
                         \label{estimate0}
\end{align}

We shall now estimate the local maximum norm error defined on $\mathcal{S}$. Let
$\mathcal{S} \csubset D_0 \csubset D_{0}^{'}\csubset D_1 \csubset \Omega_r$. By
noting local maximum norm estimate \eqref{whahlbin-estimate} in Lemma~\ref{lemma-Whahlbin}, we have
\begin{align}
  \|\xi-\xi_h\|_{L^\infty(\mathcal{S})}
  & \le C \left( h^{k+1}\left|\ln h\right|^{\delta^{k,1}} \|\xi\|_{W^{k+1,\infty}(D_1)}
    + \|\xi-\xi_h\|_{H^{-p}(D_0)}  \right)
    \nn                                                                                    \\
  & \le C \left( h^{k+1}\left|\ln h\right|^{\delta^{k,1}} \|\xi\|_{H^2(\Omega)}
    + \|\xi-\xi_h\|_{H^{-p}(D_0)}  \right)
    \nn                                                                                    \\
  & \le C \left( h^{k+1}\left|\ln h\right|^{\delta^{k,1}} \|u-u_h\|_{L^2(\Omega_r)}
    + \|\xi-\xi_h\|_{H^{-p}(D_0)}  \right),
    \label{whahlbin-estimate-5}
\end{align}
where we have used the interior regularity away from the source in Lemma \eqref{lemma-evans}.
Thus, it remains to provide a local negative norm estimate. It is clear that
\[
  \|\xi-\xi_h\|_{H^{-p}(D_0)} \le \|\xi-\xi_h\|_{L^{2}(\Omega)}
\]
By classical Aubin--Nitsche duality argument, there holds
\begin{align}
  \|\xi-\xi_h\|_{L^{2}(\Omega)} \le C h^{2 s} \|\xi\|_{H^{1+s}(\Omega)} \le C h^{2 s} \|u-u_h\|_{L^2(\Omega_r)}.
  \label{case-bad}
\end{align}
Taking the above estimate into \eqref{whahlbin-estimate-5}, we can derive
\begin{align}
  \|\xi-\xi_h\|_{L^\infty(\mathcal{S})}
  \le C \left( h^{k+1} \left|\ln h\right|^{\delta^{k,1}}  + h^{2 s} \right) \|u-u_h\|_{L^2(\Omega_r)}
  \label{whahlbin-estimate-5.1}
\end{align}
which leads to the desired local $L^2$ estimate
\begin{align}
  \|u-u_h\|_{L^2(\Omega_r)} \le C(h^{k+1}\left|\ln h\right|^{\delta^{k,1}}  + h^{2 s})\|\mu\|_{\mathcal{M}(\Omega)}
  \label{estimate0-end}
\end{align}
The local $H^1$ estimate can be obtained by Lemma~\ref{lemma-local-energy-Demlow} directly. \endproof

% --------------------------------------------------------- the case H^3
% regularity holds ---------------------------------------------------------

In several cases, the local error estimates away from the support of the
singular measure can be higher.

\begin{corollary}
  \label{coro:interior}
  If the domain $\Omega$ is a convex polygon with all interior angles less than or
  equal to $\pi/2$ in two-dimensional space, or $\Omega$ is a convex polyhedron with
  all interior edge angles less than or equal to $\pi/2$ in three-dimensional
  space, then there holds for sufficiently small $h>0$ and $k=1,2$
  \begin{empheq}[left=\empheqlbrace]{align}
    & \|u-u_h\|_{L^2(\Omega_r)} \le C h^{k+1} |\ln h|^{\delta^{k,1}}\|\mu\|_{\mathcal{M}(\Omega)}\,,
    \label{local-L2-soso}
    \\
    & \|u-u_h\|_{H^1(\Omega_r)} \le C h^k \|\mu\|_{\mathcal{M}(\Omega)}\,.
    \label{local-H1-soso}
  \end{empheq}
  Furthermore, if $\Omega$ is a rectangle or equilateral triangle in 2D or a cube in
  3D, then \eqref{local-L2-soso} and \eqref{local-H1-soso} hold for any $k \in \mathbb{N}$.
\end{corollary}

\noindent {\bf Proof}. We first prove the local $L^2$ estimate \eqref{local-L2-soso}. Recall that
$\xi$ is the solution of the following elliptic problem
\begin{align}
  -\Div(A \Grad \xi) =(u-u_h) \, {\mathcal{X}}_{\scriptscriptstyle \Omega_r}, \quad \textrm{in $\bm{x} \in \Omega$},
  \qquad
  \xi = 0~ \textrm{ on $\partial \Omega$},
\end{align}
A key step is to provide a higher convergence rate for $\|\xi-\xi_h\|_{H^{-p}(D_0)}$
on the right-hand side of \eqref{whahlbin-estimate-5}. To this end, let $f \in H_0^p(D_0)$ satisfying
$\|f\|_{H_0^p(D_0)}=1$ and
\begin{align}
  \|\xi-\xi_h\|_{H^{-p}(D_0)} = (\xi-\xi_h,f).
\end{align}
By the Sobolev extension theorem, there exists an extension
$\tilde{f}\in H_0^p(\Omega)$ such that $\tilde{f}|_{D_0}=f$ and
$\|\tilde{f}\|_{H^p(\Omega)}\le C\|f\|_{H^p(D_0)}$. Hence, we consider a dual problem
with interior support
\begin{align}
  -\Div(A\Grad \theta) =\tilde{f}, \quad \textrm{for $\bm{x} \in \Omega$},
  \qquad
  \theta = 0~ \textrm{ on $\partial \Omega$},
  \label{dual-pde-local-end}
\end{align}
and the corresponding FEM
\begin{align}
  a(\theta_h, \omega_h) = (\tilde{f},\omega_h), \quad \forall \omega_h \in V_{h,0}\,.
  \label{dual-fem-end}
\end{align}
Taking the inner product of both side of \eqref{dual-pde-local-end} with $\xi-\xi_h$ over $\Omega$ yields
\begin{align}
  \|\xi-\xi_h\|_{H^{-p}(D_0)} = (\xi-\xi_h, f)
  & = (\xi, -\Div(A\Grad \theta)) - a(\xi_h, \theta_h)
    \nn                                                                            \\
  & =(u-u_h, \theta \mathcal{X}_{\scriptscriptstyle \Omega_r})
    - (u-u_h, \theta_h {\mathcal{X}}_{\scriptscriptstyle \Omega_r})
    \nn                                                                            \\
  & =(u-u_h, (\theta - \theta_h) {\mathcal{X}}_{\scriptscriptstyle \Omega_r})
    \nn                                                                            \\
  & \le \|u-u_h\|_{L^2(\Omega_r)} \, \|\theta - \theta_h\|_{L^2(\Omega_r)}.
\end{align}
By noting the fact that $H^3$ regularity holds for $\Omega$ with interior source in
Lemma~\ref{lemma-interior-source}, we have
\begin{align}
  \|\theta - \theta_h\|_{L^2(\Omega_r)}
  \le C h^{k+1} \|\theta\|_{H^{k+1}(\Omega)}
  \le C h^{k+1} \|f\|_{H^{1}(\Omega)},
  \qquad \textrm{for $k=1$, $2$.}
\end{align}
Therefore, there holds
\begin{align}
  \|\xi-\xi_h\|_{H^{-1}(D_0)} \le C h^{k+1} \|u-u_h\|_{L^2(\Omega_r)},
  \qquad \textrm{for $k=1$, $2$.}
\end{align}
As a result, the local maximum norm estimate can be improved to
\begin{align}
  \|\xi-\xi_h\|_{L^\infty(\mathcal{S})} \le C h^{k+1}\left|\ln
  h\right|^{\delta^{k,1}}\|\xi\|_{W^{k+1,\infty}(D_1)} + C h^{k+1} \|u-u_h\|_{L^2(\Omega_r)},
  \quad  \textrm{for $k=1$, $2$},
  \label{whahlbin-estimate-6}
\end{align}
which implies that \eqref{local-L2-soso} holds.

Furthermore, if $\Omega$ is a rectangle or equilateral triangle in 2D or a cube in
3D, then the solution $\theta$ to the problem \eqref{dual-pde-local-end} admits enough regularity. That is,
we have for any $p \ge 0$
\begin{align}
  \|\theta\|_{H^{p+2}(D_1)} \le C \|\tilde{f}\|_{H^p(\Omega)}   \le C \|f\|_{H^p(D_0)},
  \qquad \text{for any} ~ D_0 \csubset D_1 \csubset \Omega.
\end{align}
Hence, we have for any $k\in \mathbb{N}$
\begin{align}
  \|\theta - \theta_h\|_{L^2(\Omega_r)}
  \le C h^{k+1} \|\theta\|_{H^{k+1}(\Omega)}
  \le C h^{k+1} \|f\|_{H^{k-1}(\Omega)}.
\end{align}
Correspondingly, there holds for any $k \in \mathbb{N}$
\begin{align}
  \|\xi-\xi_h\|_{H^{-p}(D_0)} \le C h^{k+1} \|u-u_h\|_{L^2(\Omega_r)}.
\end{align}
Taking the above estimate into \eqref{whahlbin-estimate} gives that for any $k$-th Lagrange FEM
\begin{align}
  \|\xi-\xi_h\|_{L^\infty(\mathcal{S})}
  \le C h^{k+1}\left|\ln h\right|^{\delta^{k,1}} \|u-u_h\|_{L^2(\Omega_r)}
  + C h^{k+1}  \|u-u_h\|_{L^2(\Omega_r)},
  \label{whahlbin-estimate-7}
\end{align}
which implies the desired $L^2$ estimate \eqref{local-L2-soso}.

Also, the local $H^1$ estimates \eqref{local-H1-soso} follows from the local energy estimates in
Lemma~\ref{lemma-local-energy-Demlow}. \endproof

\section{Numerical Experiments}
\label{numerical-experiments}

In this section, we present several numerical experiments to illustrate our
theoretical results. All computations are carried out using the free software
packages FEniCSx \cite{BarattaEtal2023} and GetFEM \cite{Renard-Poulios}. The meshes are generated with Gmsh \cite{geuzaineGmsh3DFinite2009}. In
all tests, we consider the problem
\begin{equation}
  -\Delta u = \mu \quad \text{in } \Omega, \qquad
  u = 0 \quad \text{on } \partial \Omega.
  \label{eq:pde-in-numerical}
\end{equation}
Here, $\mu$ is a point source in Examples~\ref{example-Lshape} and \ref{example:hexagon-domain}, and a line source in
Example~\ref{example-3D}. Since the exact solution is not available, the solution computed on
a sufficiently fine mesh is taken as a reference solution for the evaluation of
convergence rates. We denote by $N_{\text{ref}}$ the level of uniform refinement
with respect to the initial mesh. We note that the regularity of the exact
solution away from the source increases from Examples~\ref{example-Lshape} to \ref{example-3D}, and
accordingly the observed local convergence rates rise across these examples.

\begin{example}[Point source on a non-convex polygon]
  \label{example-Lshape}
  \rm In this example we solve the model problem \eqref{eq:pde-in-numerical} on the non-convex L-shaped
  domain $\Omega = (-1,1)^{2}\setminus [0,1)\times (-1,0]$, with a point source
  $\mu = \delta_{(-\frac{1}{2},\frac{1}{2})}$. We define two subdomains around the
  source by
  \begin{align}
    B_{1} &= \{\bm{x} \in \Omega : \operatorname{dist}(\bm{x},(-\tfrac{1}{2},\tfrac{1}{2})) < {1}/{6} \},\\
    B_{2} &= \{\bm{x} \in \Omega : \operatorname{dist}(\bm{x},(-\tfrac{1}{2},\tfrac{1}{2})) < {1}/{10} \}.
  \end{align}
  To test convergence, we discretize with standard Lagrange finite elements of
  degree $k=1,2,3$ on successively refined meshes. The initial mesh has 65 nodes
  and 96 elements. A numerical solution with 197633 nodes (after 6 refinements)
  computed by linear FEM is used as the exact solution. Table~\ref{tab::lshape-all-vertical} reports global
  errors and local errors on subdomains $\Omega\setminus B_1$ and $\Omega\setminus B_2$ for $k=1$. The
  global $L^{2}$ error behaves like $O(h)$ for $k=1$, in agreement with our
  theoretical prediction. Since the re-entrant angle is $\Theta = 3\pi/2$ and the
  regularity index is $s = 2/3$ for the L-shape, the predicted local convergence
  rates are $O(h^{4/3})$ in the $L^2$ norm and $O(h^{2/3})$ in the $H^1$
  semi-norm, see estimates \eqref{local-L2-linear} and \eqref{local-H1-linear} in Theorem \ref{main-thm-local}. The computed local errors
  in Table~\ref{tab::lshape-all-vertical} seems to exhibit higher rates. However, this apparent improvement
  should be treated cautiously. To further verify local behavior near the
  re-entrant corner, we introduce the subdomain
  \begin{equation}
    B_{3}=\{\bm{x}\in \Omega : \operatorname{dist}(\bm{x},(0,0)) < {1}/{6}\}.
  \end{equation}
  This choice reduces the influence of global discretization errors. We show the
  errors on the subdomain $B_{3}$ in Table~\ref{tab::lshape-l2-h1-corner} and observe local convergence
  rates on $B_{3}$ with $O(h^{4/3})$ in the $L^2$ norm and $O(h^{2/3})$ in the
  $H^1$ norm, respectively. Clearly, the local errors listed in Table~\ref{tab::lshape-l2-h1-corner}
  confirm that the estimates of Theorem~\ref{main-thm-local} are sharp for non-convex polygonal
  domains.
  
  \begin{table}[htb]
    \centering
    \caption{Computed errors in both global and local norms on the L-shape
      domain for $k=1$. (Example~\ref{example-Lshape})}
    \label{tab::lshape-all-vertical}
    \resizebox{\linewidth}{!}{%
      \begin{tabular}{?c^cc|cc|cc|cc|cc?}
        \hhrule
        $N_{\text{ref}}$
        & $\|\cdot\|_{L^2(\Omega)}$ & Rate
        & $\|\cdot\|_{L^2(\Omega \setminus \overline{B}_{1})}$ & Rate
        & $\|\cdot\|_{L^2(\Omega \setminus \overline{B}_{2})}$ & Rate
        & $|\cdot|_{H^1(\Omega \setminus \overline{B}_{1})}$ & Rate
        & $|\cdot|_{H^1(\Omega \setminus \overline{B}_{2})}$ & Rate \\
        \shrule
        0 & 1.76e-02 & --- & 1.07e-02 & --- & 1.19e-02 & --- & 2.35e-01 & --- & 2.66e-01 & --- \\
        1 & 8.63e-03 & 1.03 & 3.09e-03 & 1.79 & 4.41e-03 & 1.44 & 1.18e-01 & 0.99 & 2.05e-01 & 0.38 \\
        2 & 4.26e-03 & 1.02 & 8.34e-04 & 1.89 & 1.24e-03 & 1.83 & 5.76e-02 & 1.03 & 9.90e-02 & 1.05 \\
        3 & 2.07e-03 & 1.04 & 2.50e-04 & 1.74 & 3.36e-04 & 1.89 & 2.82e-02 & 1.03 & 4.54e-02 & 1.12 \\
        4 & 9.65e-04 & 1.10 & 7.84e-05 & 1.68 & 9.58e-05 & 1.81 & 1.41e-02 & 1.00 & 2.22e-02 & 1.03 \\
        \hhrule
      \end{tabular}}
  \end{table}

  \begin{table}[htb]
    \centering \small \setlength{\tabcolsep}{4pt}
    \caption{Local errors in the $L^2(B_{3})$ norm and $H^1(B_{3})$ semi-norm
      for $k=1,2,3$ on the L-shape domain. (Example~\ref{example-Lshape})}
    \label{tab::lshape-l2-h1-corner}
    \begin{tabular}{?c^cc|cc|cc^cc|cc|cc?}
      \hhrule
      & \multicolumn{6}{c^}{\textbf{$\bm{L^2(B_3)}$ norm}}
      & \multicolumn{6}{c?}{\textbf{$\bm{H^1(B_3)}$ semi-norm}} \\
      \shrule
      $N_{\text{ref}}$
      & $k=1$ & Rate & $k=2$ & Rate & $k=3$ & Rate
                                    & $k=1$ & Rate & $k=2$ & Rate & $k=3$ & Rate \\
      \hline
      0
      & 1.80e-03 & ---  & 4.20e-04 & ---  & 1.85e-04 & ---
                                    & 2.61e-02 & ---  & 1.48e-02 & ---  & 9.16e-03 & ---  \\
      1
      & 7.04e-04 & 1.36 & 1.79e-04 & 1.23 & 7.15e-05 & 1.37
                                    & 1.99e-02 & 0.39 & 9.57e-03 & 0.63 & 6.05e-03 & 0.60 \\
      2
      & 3.01e-04 & 1.23 & 6.85e-05 & 1.39 & 2.71e-05 & 1.40
                                    & 1.33e-02 & 0.58 & 6.00e-03 & 0.67 & 3.79e-03 & 0.68 \\
      3
      & 1.22e-04 & 1.31 & 2.57e-05 & 1.42 & 1.02e-05 & 1.42
                                    & 8.51e-03 & 0.65 & 3.71e-03 & 0.69 & 2.34e-03 & 0.69 \\
      4
      & 4.52e-05 & 1.43 & 9.06e-06 & 1.50 & 3.59e-06 & 1.50
                                    & 5.17e-03 & 0.72 & 2.21e-03 & 0.74 & 1.40e-03 & 0.74 \\
      \hhrule
    \end{tabular}
  \end{table}

\end{example}

\begin{example}[Point source on a convex polygon]
  \label{example:hexagon-domain}
  \rm

  In this example we solve the model problem \eqref{eq:pde-in-numerical} with a point source
  $\mu=\delta_{(0,0)}$ concentrated at the origin $O$. The computational domain is a
  hexagon with vertices $V_1,\dots,V_6$ listed in order:
  \[
    \begin{bmatrix}
      V_1 & V_2 & V_3 & V_4 & V_5 & V_6\\
      (-\tfrac{1}{\sqrt{3}},1) & (\tfrac{1}{\sqrt{3}},1) & (\tfrac{2}{\sqrt{3}},0) &
                                                                                     (\tfrac{1}{\sqrt{3}},-1) & (-\tfrac{1}{\sqrt{3}},-1) & (-\tfrac{1}{\sqrt{3}}-\tfrac{1}{10},0)
    \end{bmatrix}.
  \]
  The interior angle near $V_6$ is close to $\pi$. The domain, initial mesh and a
  numerical solution on a fine mesh are shown in Figure~\ref{fig:example-hexagon}. We discretize with
  linear, quadratic and cubic Lagrange elements ($k=1,2,3$) on successively
  refined meshes. A numerical solution after 6 refinements computed by
  $k$th-order Lagrange FEM is used as the exact solution. The computed errors
  are reported in Table~\ref{tab:hexagon-k123}. From Table~\ref{tab:hexagon-k123} we observe that the global $L^{2}$
  error behaves like $O(h)$ for $k=1,2,3$, which agrees with Theorem~\ref{thm:global-error}.

  \begin{figure}[h]
    \centering
    \begin{minipage}[c]{0.34\linewidth}
      \begin{tikzpicture}[scale=1.6, thick]
        \pgfmathsetmacro{\invsqrt}{0.57735} % 1/sqrt(3) ≈ 0.57735
        \pgfmathsetmacro{\twoinvsqrt}{1.15470} % 2/sqrt(3) ≈ 1.15470
        \pgfmathsetmacro{\specialPoint}{-0.67735} % -2/(2√3)-0.1 ≈ -0.67735

        \coordinate (1) at (-\invsqrt, 1.0); \coordinate (2) at (\invsqrt, 1.0);
        \coordinate (3) at (\twoinvsqrt, 0.0); \coordinate (4) at (\invsqrt,
        -1.0); \coordinate (5) at (-\invsqrt, -1.0); \coordinate (6) at
        (\specialPoint, 0.0); \coordinate (-1) at ({-\twoinvsqrt}, 0.0);

        \draw[black,thick, fill=green!20, opacity=1.0] (1) -- (2) -- (3) -- (4)
        -- (5) -- (6) -- cycle; \draw[black,dashed] (1) -- (-1) -- (5);

        \begin{scope}
          \clip (1) -- (2) -- (3) -- (4) -- (5) -- (6) -- cycle; \draw[red] (6)
          circle (0.16667);
        \end{scope}

        \draw[->] (-1.5,0) -- (1.5,0) node[right] {$x$}; \draw[->] (0,-1.2) --
        (0,1.2) node[above] {$y$};

        \foreach \point/\position/\name in { 1/above right/1, 2/right/2, 3/below
          right/3, 4/below left/4, 5/left/5, 6/above right/6} { \fill[black]
          (\point) circle (0.8pt); \node[\position, black] at (\point)
          {$V_{\name}$}; } \fill[red] (0.0,0.0) circle (1.2pt);

        \foreach \x in {-1,-0.5,0.5,1} \draw (\x,0.05) -- (\x,-0.05) node[below]
        {$\x$}; \foreach \y in {-1,-0.5,0.5,1} \draw (0.05,\y) -- (-0.05,\y)
        node[left] {$\y$}; \node[below right] at (0,0) {$O$};

      \end{tikzpicture}
    \end{minipage}
    \hfill
    \begin{minipage}[c]{0.2\linewidth}
      \centering
      \includegraphics[width=0.95\textwidth]{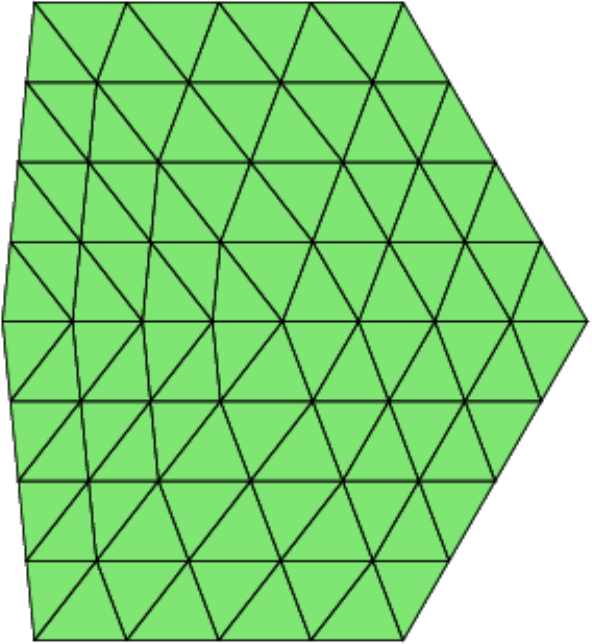}
    \end{minipage}
    \hfill
    \begin{minipage}[c]{0.42\linewidth}
      \centering
      \includegraphics[width=0.99\textwidth]{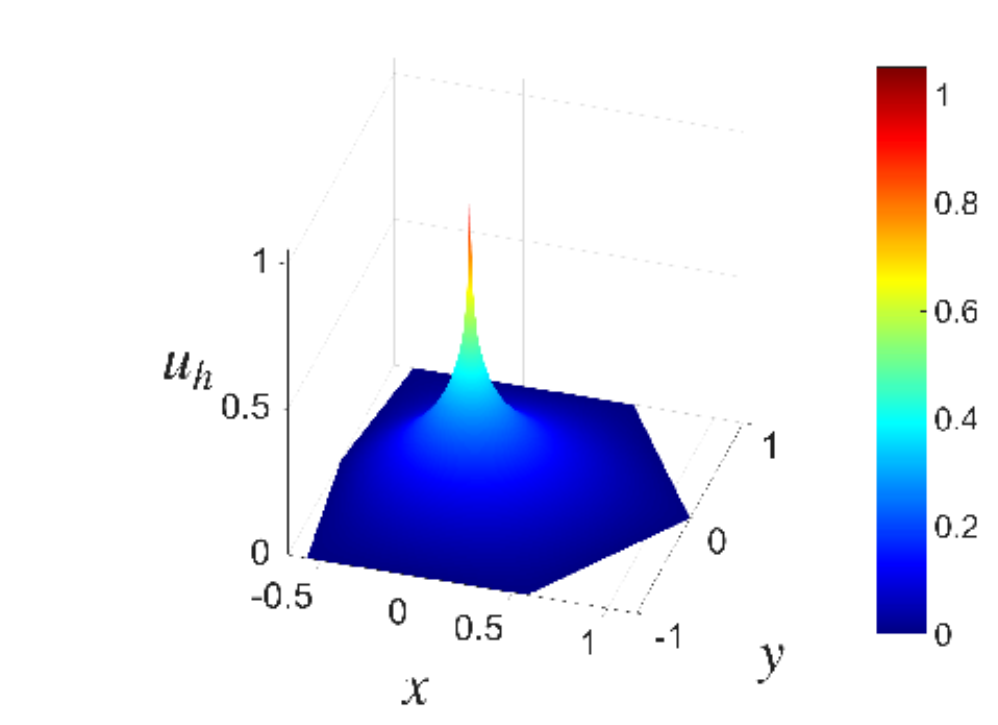}
    \end{minipage}
    \caption{The domain (left); initial mesh (middle); numerical solution on a
      fine mesh using linear FEM (right). (Example~\ref{example:hexagon-domain})}
    \label{fig:example-hexagon}
  \end{figure}
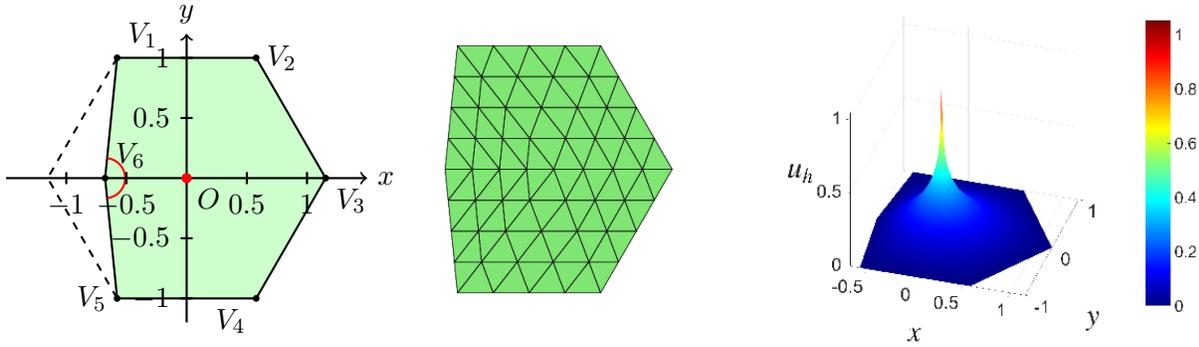
  
  To verify local convergence rates, we compute $L^{2}$ and $H^{1}$ errors on
  $\Omega\setminus B_1$ and $\Omega\setminus B_2$, where
  \begin{align}
    B_{1} &= \{\bm{x} \in \Omega : \operatorname{dist}(\bm{x},(0,0)) < {1}/{6} \},\\
    B_{2} &= \{\bm{x} \in \Omega : \operatorname{dist}(\bm{x},(0,0)) < {1}/{10} \}.
  \end{align}
  Although Table~\ref{tab:hexagon-k123} seems to exhibit local convergence rates exceeding second
  order for $k=2,3$, this behavior should be interpreted with care. To assess
  the sharpness of the local estimate in general convex polygonal domains, we
  restrict the error measurement to a neighborhood of the vertex $V_6$ and
  define
  \begin{equation}
    B_{3}=\{\bm{x}\in \Omega : \operatorname{dist}(\bm{x},V_{6}) < 1/6\}.
  \end{equation}
  The corresponding errors $\|u-u_{h}\|_{L^{2}(B_{3})}$ and
  $|u-u_{h}|_{H^{1}(B_{3})}$ are reported in Table~\ref{tab::hexagon-l2-h1-corner}. In this setting, the
  observed convergence rates are $O(h^{2})$ in the $L^{2}$ norm and $O(h)$ in
  the $H^{1}$ semi-norm, independently of the polynomial degree $k$. These rates
  coincide with the prediction of Theorem~\ref{main-thm-local} and demonstrate that the local
  error behavior is ultimately limited by the corner singularity. Hence, the
  numerical results confirm the sharpness of the theoretical estimates. We note
  that the numerical results reported in \cite{Lee-Choi2020} are insufficient to demonstrate
  their optimal local convergence results, since corner singularities of convex
  polygons were not taken into account.

  \begin{table}[htb]
    \centering
    \caption{Computed errors on the hexagon with $k=1,2,3$. (Example~\ref{example:hexagon-domain})}
    \label{tab:hexagon-k123}
    \resizebox{\linewidth}{!}{%
      \begin{tabular}{?c^cc|cc|cc|cc|cc?}
        \hhrule

        \multicolumn{11}{?c?}{\textbf{Linear FEM ($\bm{k=1}$)}}\\
        \shrule
        $N_{\text{ref}}$
        & $\|\cdot\|_{L^2(\Omega)}$ & Rate
        & $\|\cdot\|_{L^2(\Omega \setminus \overline{B}_{1})}$ & Rate
        & $\|\cdot\|_{L^2(\Omega \setminus \overline{B}_{2})}$ & Rate
        & $|\cdot|_{H^1(\Omega \setminus \overline{B}_{1})}$ & Rate
        & $|\cdot|_{H^1(\Omega \setminus \overline{B}_{2})}$ & Rate \\
        \hline
        0 & 1.73e-02 & --- & 8.01e-03 & --- & 1.01e-02 & --- & 2.07e-01 & --- & 2.40e-01 & --- \\
        1 & 1.01e-02 & 0.77 & 2.18e-03 & 1.88 & 3.57e-03 & 1.50 & 1.01e-01 & 1.03 & 1.80e-01 & 0.41 \\
        2 & 6.07e-03 & 0.74 & 6.00e-04 & 1.86 & 1.05e-03 & 1.76 & 4.96e-02 & 1.03 & 8.41e-02 & 1.10 \\
        3 & 2.58e-03 & 1.24 & 1.42e-04 & 2.08 & 2.39e-04 & 2.14 & 2.40e-02 & 1.04 & 4.01e-02 & 1.07 \\
        4 & 1.39e-03 & 0.89 & 3.69e-05 & 1.94 & 6.23e-05 & 1.94 & 1.18e-02 & 1.03 & 1.94e-02 & 1.05 \\

        \hhrule

        \multicolumn{11}{?c?}{\textbf{Quadratic FEM ($\bm{k=2}$)}}\\
        \shrule
        $N_{\text{ref}}$
        & $\|\cdot\|_{L^2(\Omega)}$ & Rate
        & $\|\cdot\|_{L^2(\Omega \setminus \overline{B}_{1})}$ & Rate
        & $\|\cdot\|_{L^2(\Omega \setminus \overline{B}_{2})}$ & Rate
        & $|\cdot|_{H^1(\Omega \setminus \overline{B}_{1})}$ & Rate
        & $|\cdot|_{H^1(\Omega \setminus \overline{B}_{2})}$ & Rate \\
        \hline
        0 & 8.64e-03 & --- & 3.15e-03 & --- & 4.26e-03 & --- & 1.00e-01 & --- & 1.52e-01 & --- \\
        1 & 5.00e-03 & 0.79 & 3.15e-04 & 3.32 & 1.22e-03 & 1.80 & 1.90e-02 & 2.40 & 7.25e-02 & 1.07 \\
        2 & 2.35e-03 & 1.09 & 3.33e-05 & 3.24 & 9.73e-05 & 3.65 & 4.26e-03 & 2.16 & 1.20e-02 & 2.60 \\
        3 & 1.35e-03 & 0.80 & 4.27e-06 & 2.96 & 1.09e-05 & 3.15 & 1.07e-03 & 2.00 & 2.79e-03 & 2.10 \\
        4 & 5.24e-04 & 1.37 & 5.18e-07 & 3.04 & 1.25e-06 & 3.12 & 2.79e-04 & 1.93 & 6.90e-04 & 2.01 \\

        \hhrule

        \multicolumn{11}{?c?}{\textbf{Cubic FEM ($\bm{k=3}$)}}\\
        \shrule
        $N_{\text{ref}}$
        & $\|\cdot\|_{L^2(\Omega)}$ & Rate
        & $\|\cdot\|_{L^2(\Omega \setminus \overline{B}_{1})}$ & Rate
        & $\|\cdot\|_{L^2(\Omega \setminus \overline{B}_{2})}$ & Rate
        & $|\cdot|_{H^1(\Omega \setminus \overline{B}_{1})}$ & Rate
        & $|\cdot|_{H^1(\Omega \setminus \overline{B}_{2})}$ & Rate \\
        \hline
        0 & 5.74e-03 & --- & 1.34e-03 & --- & 2.27e-03 & --- & 8.26e-02 & --- & 1.17e-01 & --- \\
        1 & 2.98e-03 & 0.95 & 4.34e-05 & 4.95 & 3.33e-04 & 2.77 & 5.37e-03 & 3.94 & 3.68e-02 & 1.67 \\
        2 & 1.53e-03 & 0.96 & 2.52e-06 & 4.10 & 1.15e-05 & 4.85 & 4.81e-04 & 3.48 & 1.96e-03 & 4.23 \\
        3 & 7.52e-04 & 1.03 & 2.26e-07 & 3.48 & 5.94e-07 & 4.28 & 1.16e-04 & 2.06 & 2.34e-04 & 3.07 \\
        4 & 4.01e-04 & 0.91 & 4.24e-08 & 2.41 & 5.83e-08 & 3.35 & 4.78e-05 & 1.27 & 5.40e-05 & 2.12 \\

        \hhrule
      \end{tabular}}
  \end{table}

  \begin{table}[htb]
    \centering
    \caption{Local errors in the $L^2(B_{3})$ norm and $H^1(B_{3})$ semi-norm
      for $k=1,2,3$. (Example~\ref{example:hexagon-domain})}
    \label{tab::hexagon-l2-h1-corner}
    \resizebox{\linewidth}{!}{%
      \begin{tabular}{?c^cc|cc|cc^cc|cc|cc?}
        \hhrule
        & \multicolumn{6}{c^}{\textbf{$\bm{L^2(B_3)}$ norm}}
        & \multicolumn{6}{c?}{\textbf{$\bm{H^1(B_3)}$ semi-norm}} \\
        \shrule
        $N_{\text{ref}}$
        & $k=1$ & Rate & $k=2$ & Rate & $k=3$ & Rate
                                      & $k=1$ & Rate & $k=2$ & Rate & $k=3$ & Rate \\
        \hline
        0
        & 6.61e-04 & ---  & 8.06e-05 & ---  & 1.04e-05 & ---
                                      & 8.43e-03 & ---  & 2.63e-03 & ---  & 8.15e-04 & ---  \\
        1
        & 1.13e-04 & 2.55 & 1.32e-05 & 2.61 & 2.74e-06 & 1.93
                                      & 3.94e-03 & 1.10 & 1.07e-03 & 1.29 & 4.25e-04 & 0.94 \\
        2
        & 3.12e-05 & 1.86 & 2.63e-06 & 2.32 & 6.62e-07 & 2.05
                                      & 2.03e-03 & 0.96 & 4.64e-04 & 1.21 & 2.04e-04 & 1.06 \\
        3
        & 7.60e-06 & 2.04 & 5.86e-07 & 2.17 & 1.58e-07 & 2.06
                                      & 1.08e-03 & 0.91 & 2.14e-04 & 1.12 & 9.69e-05 & 1.07 \\
        4
        & 1.93e-06 & 1.97 & 1.36e-07 & 2.11 & 3.78e-08 & 2.06
                                      & 5.57e-04 & 0.95 & 9.91e-05 & 1.11 & 4.53e-05 & 1.10 \\
        \hhrule
      \end{tabular}}
  \end{table}

\end{example}

\begin{example}[Line source on a cube]
  \label{example-3D}
  \rm In the final example we consider the 3D--1D model \eqref{eq:pde-in-numerical} on the unit cube
  $\Omega=(0,1)^3$. The right-hand side is a line measure supported on parametrized
  curves $\Lambda\subset\mathbb{R}^3$. The distribution $\delta_{\Lambda}$ is defined by
  \begin{equation}
    \langle\delta_{\Lambda},v\rangle=\int_{\Lambda} v\,\mathrm{d}s,
    \qquad \forall v\in C(\Omega).
  \end{equation}
  We introduce the scalar functions
  \begin{align}
    f_1(t) &= 0.5 + 0.1(1-t)\sin(4\pi t), \\
    f_2(t) &= 0.5 + 0.1(1-t)\cos(4\pi t), \\
    f_3(t) &= 0.3 + t,
             \label{eq::line-function-part}
  \end{align}
  and define three parametric curves
  \begin{align}
    s_1(t) &= (f_1(t),f_2(t),f_3(t)), \\
    s_2(t) &= (f_2(t),f_3(t),f_1(t)), \\
    s_3(t) &= (f_3(t),f_1(t),f_2(t)),
  \end{align}
  for $t\in[0,0.4]$. Denote $\Lambda_i=s_i([0,0.4])$ for $i=1,2,3$, and set
  \begin{align}
    B_1 &= \{\bm{x}\in\Omega:\operatorname{dist}(\bm{x},(0.5,0.5,0.5))<0.3\},\\
    B_2 &= \{\bm{x}\in\Omega:\operatorname{dist}(\bm{x},(0.5,0.5,0.5))<0.4\}.
  \end{align}

   \begin{figure}[hbt]
    \centering
    \begin{minipage}[c]{0.33\linewidth}
      \includegraphics[width=0.99\textwidth]{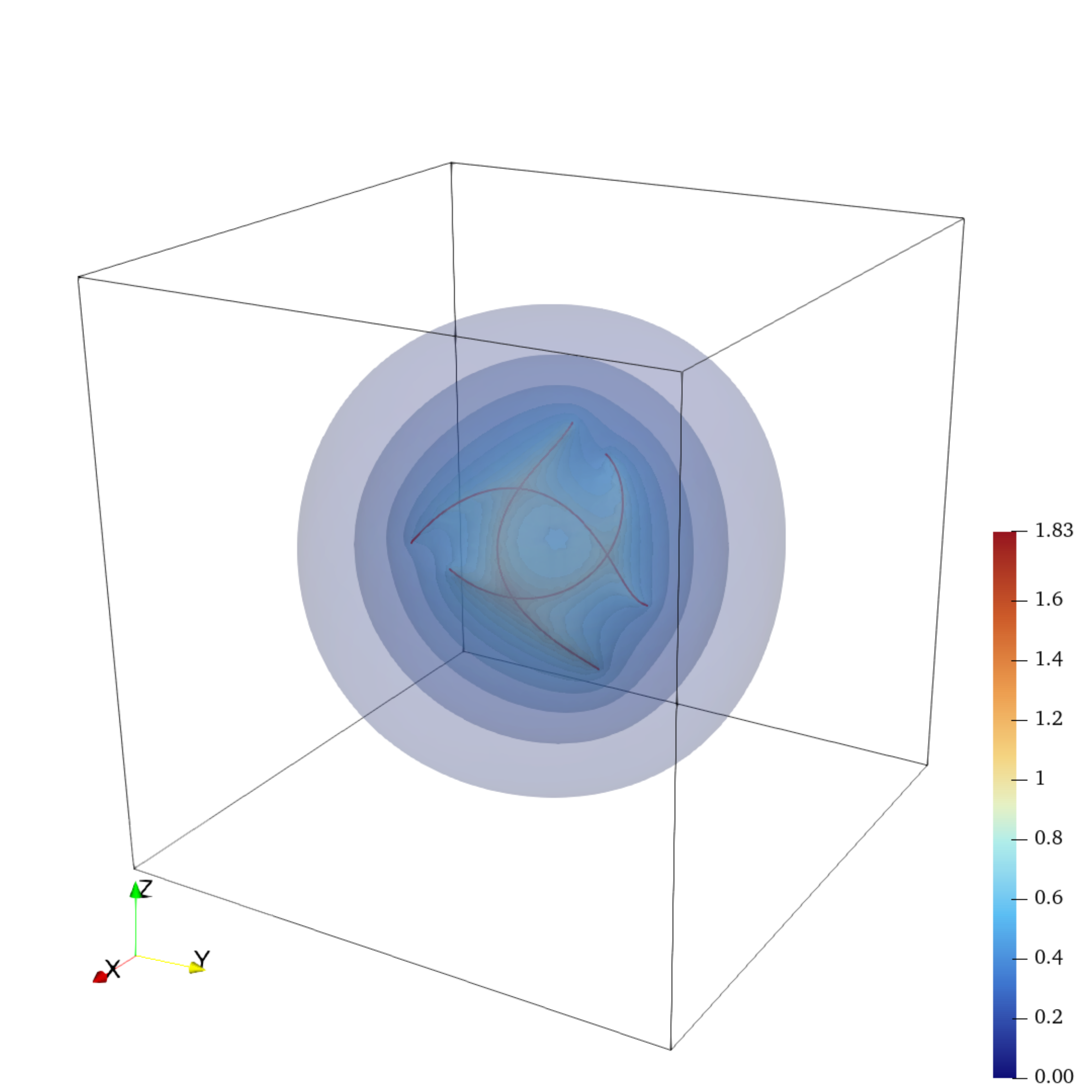}
    \end{minipage}
    \begin{minipage}[c]{0.15\linewidth}
      \hfill
    \end{minipage}
    \begin{minipage}[c]{0.23\linewidth}
      \includegraphics[width=0.99\textwidth]{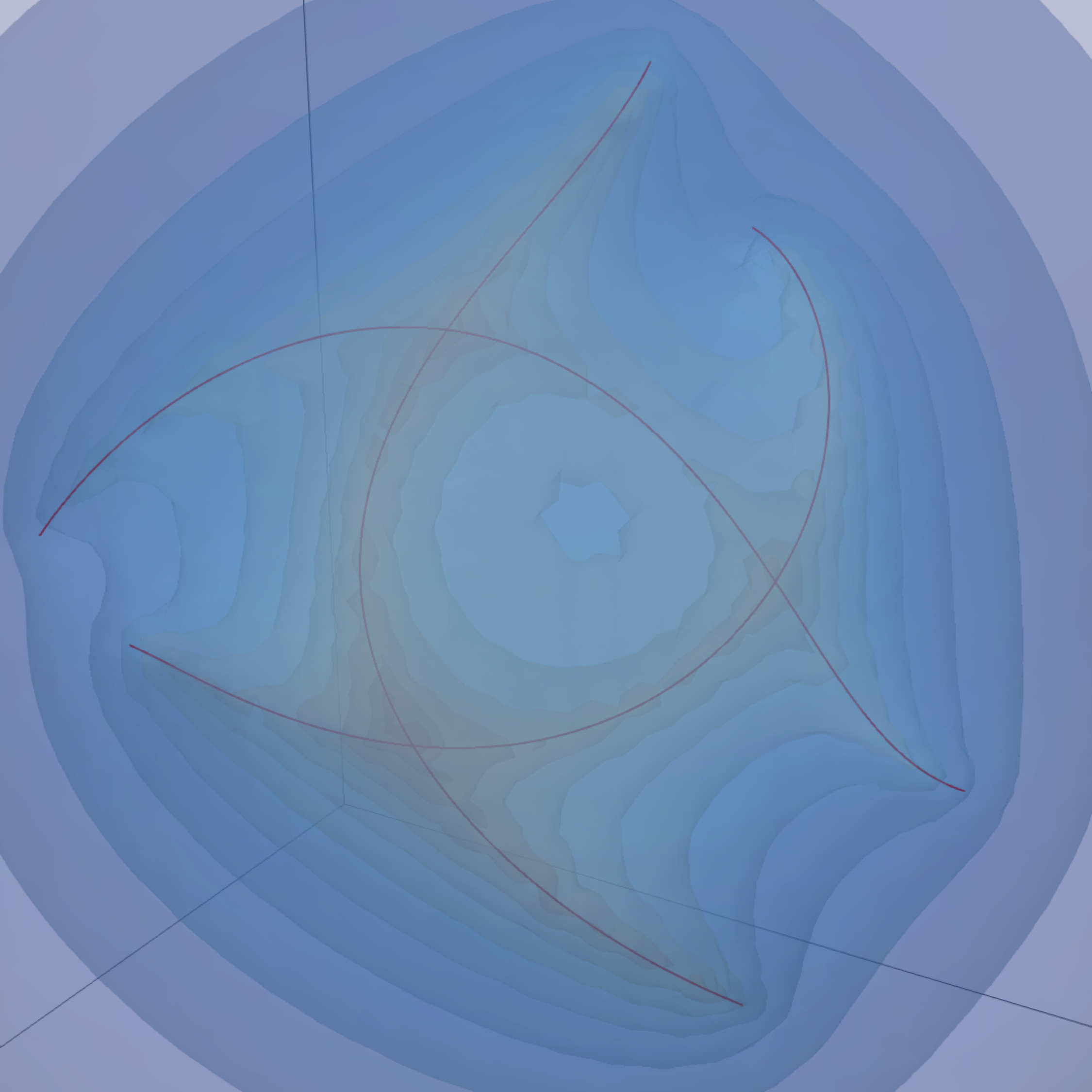}
    \end{minipage}
    \caption{Isosurface plot of the numerical solution obtained with linear
      elements ($k=1$) (left); detailed view near the sources (right).
      (Example~\ref{example-3D})}
    \label{fig:contour-curve}
  \end{figure}

  By construction $\{\Lambda_i\}_{i=1}^3\csubset B_1\csubset B_2$. Then, the measure $\mu$
  is defined by
  \begin{equation}
    \mu = 1.6\,\delta_{\Lambda_1} + 0.8\,\delta_{\Lambda_2} + 1.2\,\delta_{\Lambda_3}.
    \label{eq:mu-in-3d}
  \end{equation}

  We discretize with standard Lagrange elements of degree $k=1,2,3$ on uniformly
  refined meshes. The initial mesh has 27 nodes and 48 tetrahedra. A numerical
  solution after 5 refinements computed by $k$th-order Lagrange FEM is used as
  the exact solution. The numerical solution for $k=1$ is shown in Figure~\ref{fig:contour-curve};
  the source curves are plotted in red. The solution is concentrated near the
  curves and decays rapidly away from them, matching the expected singular
  behaviour induced by line sources.

  Errors reported in Table~\ref{tab::cube-curve} corroborate the local estimates \eqref{local-L2-soso} and \eqref{local-H1-soso} in
  Corollary~\ref{coro:interior} and indicate absence of a pollution effect. For the smooth
  parametrized curves used here the exact solution attains improved regularity
  and an $O(h)$ global rate is observed; for a general Radon measure $\mu$ such an
  improvement need not hold.

  \begin{table}[hbt]
    \centering
    \caption{Errors in both global and local norms on the cube. (Example~\ref{example-3D})}
    \label{tab::cube-curve}
    \resizebox{\linewidth}{!}{
      \begin{tabular}{?c^cc|cc|cc|cc|cc?}
        \hhrule
        \multicolumn{11}{?c?}{\textbf{Linear FEM ($\bm{k=1}$)}}\\
        \shrule
        $N_{\text{ref}}$ & $\|\cdot\|_{L^2(\Omega)}$ & Rate
        & $\|\cdot\|_{L^2(\Omega \setminus B_{1})}$ & Rate
        & $\|\cdot\|_{L^2(\Omega \setminus B_{2})}$ & Rate
        & $|\cdot|_{H^1(\Omega \setminus B_1)}$ & Rate
        & $|\cdot|_{H^1(\Omega \setminus B_2)}$ & Rate \\
        \hline
        0 & 1.21e-01 & --- & 6.05e-02 & --- & 4.01e-02 & --- & 1.53e+00 & --- & 1.43e+00 & --- \\
        1 & 7.13e-02 & 0.76 & 2.59e-02 & 1.23 & 1.82e-02 & 1.14 & 7.51e-01 & 1.03 & 6.56e-01 & 1.13 \\
        2 & 3.68e-02 & 0.95 & 8.38e-03 & 1.62 & 4.91e-03 & 1.89 & 3.74e-01 & 1.01 & 2.22e-01 & 1.56 \\
        3 & 1.40e-02 & 1.40 & 1.90e-03 & 2.14 & 1.14e-03 & 2.10 & 1.36e-01 & 1.46 & 8.59e-02 & 1.37 \\
        4 & 4.93e-03 & 1.50 & 4.07e-04 & 2.22 & 2.44e-04 & 2.23 & 5.53e-02 & 1.30 & 3.56e-02 & 1.27 \\
        \hhrule
        \multicolumn{11}{?c?}{\textbf{Quadratic FEM ($\bm{k=2}$)}}\\
        \shrule
        $N_{\text{ref}}$ & $\|\cdot\|_{L^2(\Omega)}$ & Rate
        & $\|\cdot\|_{L^2(\Omega \setminus B_{1})}$ & Rate
        & $\|\cdot\|_{L^2(\Omega \setminus B_{2})}$ & Rate
        & $|\cdot|_{H^1(\Omega \setminus B_1)}$ & Rate
        & $|\cdot|_{H^1(\Omega \setminus B_2)}$ & Rate \\
        \hline
        0 & 5.03e-02 & --- & 2.11e-02 & --- & 1.28e-02 & --- & 9.03e-01 & --- & 6.61e-01 & --- \\
        1 & 2.90e-02 & 0.79 & 4.77e-03 & 2.15 & 2.36e-03 & 2.44 & 3.31e-01 & 1.45 & 1.44e-01 & 2.19 \\
        2 & 1.35e-02 & 1.10 & 6.94e-04 & 2.78 & 2.68e-04 & 3.14 & 6.46e-02 & 2.36 & 2.35e-02 & 2.62 \\
        3 & 5.31e-03 & 1.35 & 7.71e-05 & 3.17 & 3.03e-05 & 3.14 & 1.22e-02 & 2.41 & 4.73e-03 & 2.31 \\
        4 & 2.11e-03 & 1.33 & 9.28e-06 & 3.05 & 3.71e-06 & 3.03 & 2.60e-03 & 2.22 & 1.08e-03 & 2.13 \\
        \hhrule
        \multicolumn{11}{?c?}{\textbf{Cubic FEM ($\bm{k=3}$)}}\\
        \shrule
        $N_{\text{ref}}$ & $\|\cdot\|_{L^2(\Omega)}$ & Rate
        & $\|\cdot\|_{L^2(\Omega \setminus B_{1})}$ & Rate
        & $\|\cdot\|_{L^2(\Omega \setminus B_{2})}$ & Rate
        & $|\cdot|_{H^1(\Omega \setminus B_1)}$ & Rate
        & $|\cdot|_{H^1(\Omega \setminus B_2)}$ & Rate \\
        \hline
        0 & 3.40e-02 & --- & 1.05e-02 & --- & 6.51e-03 & --- & 7.61e-01 & --- & 3.58e-01 & --- \\
        1 & 1.73e-02 & 0.98 & 2.07e-03 & 2.34 & 6.30e-04 & 3.37 & 2.41e-01 & 1.66 & 6.59e-02 & 2.44 \\
        2 & 7.75e-03 & 1.16 & 1.22e-04 & 4.09 & 2.48e-05 & 4.67 & 1.79e-02 & 3.75 & 3.17e-03 & 4.37 \\
        3 & 3.42e-03 & 1.18 & 7.46e-06 & 4.03 & 1.38e-06 & 4.17 & 1.67e-03 & 3.42 & 2.88e-04 & 3.46 \\
        4 & 1.37e-03 & 1.33 & 3.74e-07 & 4.32 & 8.28e-08 & 4.06 & 1.47e-04 & 3.51 & 3.25e-05 & 3.15 \\
        \hhrule
      \end{tabular}}
  \end{table}

\end{example}

% ----------------------------------------------------------------
\section{Conclusion}
\label{sect-Conclusion}

We have shown that standard finite element methods admit optimal local error
estimates for elliptic problems with measure-valued sources of compact support.
The framework of very weak solutions and the corresponding numerical schemes
plays a crucial role in the analysis of such problems. Our results demonstrate
that the loss of global convergence is purely local in nature. Possible
extensions of this work include adaptive finite element methods and
time-dependent problems with measure data.

% ----------------------------------------------------------------

\section*{Declarations}

\noindent {\bf The Conflict of Interest Statement:}
No conflict of interest exists. \vskip 0.1in

\noindent {\bf Availability of data and material:}
The code to reproduce the numerical results presented in this paper is available
at \url{https://github.com/bombeuler/DiracSource}. \vskip 0.1in

\noindent {\bf Funding:}
The work of the first author was supported in part by the National Key Research
and Development Program of China (No.2023YFC3804500) and National Science
Foundation of China No.12231003. \vskip 0.1in

% ----------------------------------------------------------------

\end{document}